\documentclass[a4paper,portuguese,twoside,thmsa,12pt]{article}
\usepackage{amsmath}
\usepackage{amsfonts}
\usepackage[latin1]{inputenc}
\usepackage{amssymb}
\usepackage{latexsym}
\usepackage{makeidx}
\usepackage{indentfirst}
\usepackage{amscd}
\usepackage{eufrak}
\usepackage[latin1]{inputenc}
\usepackage{amsmath}
\usepackage{amstext}
\usepackage[dvips]{graphicx}
\usepackage{graphicx,color}
\usepackage{wrapfig}
\usepackage{fancyhdr}

\DeclareMathOperator{\supp}{supp}

\newtheorem{theorem}{Theorem}

\newtheorem{corollary}[theorem]{Corollary}

\newtheorem{lemma}[theorem]{Lemma}

\newtheorem{proposition}[theorem]{Proposition}

\newenvironment{proof}[1][Proof]{\noindent\textbf{#1:} }{\hfill \rule{0.5em}{0.5em}}

\input{tcilatex}
\begin{document}

\title{Generic properties of Mañé's set of exact magnetic Lagrangians}
\author{Alexandre Rocha}
\maketitle

\begin{abstract}
Let $M$ be a closed manifold and $L$ an exact magnetic Lagrangian. In this
paper we proved that there exists a residual $\mathcal{G}$ of $H^{1}\left( M;%
\mathbb{R}\right)$ such that the pro-perty:
\begin{equation*}
{\widetilde{\mathcal{M}}}\left( c\right) ={\widetilde{\mathcal{A}}}\left(
c\right) ={\widetilde{\mathcal{N}}}\left( c\right), \forall c\in \mathcal{G}
\end{equation*}
with ${\widetilde{\mathcal{M}}}\left( c\right)$ supports on a uniquely
ergodic measure, is generic in the family of exact magnetic Lagrangians.

We also prove that, fixed the cohomology class $c$, there exists a residual
set of exact magnetic Lagrangians such that when this unique measure is
supported on a periodic orbit, this orbit is hyperbolic and its stable and
unstable manifolds intersect transversally. This result is a version of
Theorem D of \cite{gon5} for the exact magnetic Lagrangian case.
\end{abstract}

\section{Introduction}

Let $M$ be a closed manifold equipped with a Riemannian metric $%
g=\left\langle .,.\right\rangle $. A Lagrangian $L:TM\rightarrow \mathbb{R}$
is called exact magnetic Lagrangian if
\begin{equation*}
L\left( x,v\right) =\frac{\left\Vert v\right\Vert ^{2}}{2}+\eta _{x}\left(
v\right)
\end{equation*}%
for some non-closed 1-form $\eta _{x}$. The Euler-Lagrangian flow of this
Lagrangian can also be obtained as magnetic flow associated to an exact
2-form $\Omega =-d\eta _{x}.$

This type of Lagrangian fits into Mather's theory, as developed by R.
Ma\~n\'e and A. Fathi, about Tonelli Lagrangians, namely, it is fiberwise
convex and superlinear.

Let $\mathfrak{M}\left( L\right) $ be the set of action minimizing measures.
Recall that $\mathfrak{M}\left( L\right) $ is the set of $\mu $ Borel
probability measures on $TM$ which are invariant under the Euler-Lagrange
flow $\varphi _{t}$ generated by $L$ and minimizes the \textit{action}.

Since the Euler Lagrange flow generated by $L$ does not change by adding a
closed one form $\zeta $, we also consider the action minimizing measures $%
\mathfrak{M}\left( L-\zeta \right) $. The minimal action value, depends only
on the cohomology class $c=[\zeta ]\in H^{1}(M,{\mathbb{R}})$ of the closed
one form, so it is denoted by $-\alpha (c)$. It is known that $\alpha (c)$
is the energy level that contains the \textit{Mather set for the cohomology
class }$c$:%
\begin{equation*}
{\widetilde{\mathcal{M}}}_{c}\left( L\right) =\overline{\bigcup_{\mu \in
\mathfrak{M}(L-\zeta )}\supp(\mu )}.
\end{equation*}%
${\widetilde{\mathcal{M}}}_{c}\left( L\right) $ is a compact invariant set
which is a graph over a compact subset $\mathcal{M}_{c}\left( L\right) $ of $%
M$, the projected Mather set (see \cite{mat1}). $\mathcal{M}_{c}\left(
L\right) $ is laminated by curves, which are global (or time independent)
minimizers. Mather also proved that the function $c\mapsto \alpha (c)$ is
convex and superlinear.

In general, ${\widetilde{\mathcal{M}}}_{c}\left( L\right) $ is contained in
another compact invariant set, which is also a graph whose projection is
laminated by global minimizers: the \textit{Aubry set for the cohomology
class c}, denoted by $\widetilde{\mathcal{A}}_{c}\left( L\right) $.

In order to state our results, we need to introduce the Aubry set and the Mañ%
é set for a general Tonelli Lagrangian:

Let $\xi $ be a closed one form representative of the cohomology class $c=%
\left[ \xi \right] .$ The \textit{action} of a $C^{1}$ curve $\gamma
:[0,T]\rightarrow M$ is defined by

\begin{equation*}
A_{c}^{L+k}\left( \gamma \right) =\int_{0}^{T}[L(\gamma ,\dot{\gamma})-\xi
(\gamma )(\dot{\gamma})+k]dt
\end{equation*}%
where $k$ is a real number. The energy level $c\left( L-\xi \right) ,$
namely Mañé's critical value of the Lagrangian $L-c,$ which depends only on
the cohomology class $c=[\xi ]$, may be characterized in several ways. $%
c\left( L-\xi \right) $ is defined by Mañé as the infimum of the numbers $k$
such that the action $A_{c}^{L+k}\left( \delta \right) $ is nonnegative for
every closed curve $\delta :\left[ 0,T\right] \rightarrow M.$

Recall that, for a given real number $k$ the action potential $\Phi
_{c}^{L+k}:M\times M\rightarrow
\mathbb{R}
$ is defined by%
\begin{equation*}
\Phi _{c}^{L+k}\left( x,y\right) =\inf A_{c}^{L+k}\left( \gamma \right)
\end{equation*}%
infimum taken over the curves $\gamma $ joining $x$ the $y$.

Mañé proved that $c\left( L-\xi \right) =-\alpha (\left[ \xi \right]
)=-\alpha (c),$ and that $\alpha (c)$ is the smallest number such that the
action potential is finite. In other words, if $k<\alpha (c)$, then $\Phi
_{L-\xi +k}(x,y)=-\infty $ and if $k\geq \alpha (c)$, $\Phi _{L-\xi
+k}(x,y)\in {\mathbb{R}}$. We introduce the following notations: $\Phi
_{c}^{L}:=\Phi _{c}^{L+\alpha (c)}$ and $A_{c}^{L}:=A_{c}^{L+\alpha (c)}$.

Observe that by Tonelli's Therorem (See for example in \cite{gon7}), for
fixed $t>0$, there always exists a minimizing extremal curve connecting $x$
to $y$ in time $t$. The potential calculates the global (or time
independent) infimum of the action. This value may not be realized by a
curve.

The potential $\Phi _{c}^{L}$ is not symmetric in general but%
\begin{equation*}
\delta _{M}\left( x,y\right) =\Phi _{c}^{L}\left( x,y\right) +\Phi
_{c}^{L}\left( y,x\right)
\end{equation*}%
is a pseudo-metric. A curve $\gamma :%
\mathbb{R}
\rightarrow M$ is called \textit{semistatic} if minimizes action between any
of its points:%
\begin{equation*}
A_{c}^{L}\left( \gamma |_{\left[ a,b\right] }\right) =\Phi _{c}^{L}\left(
\gamma \left( a\right) ,\gamma \left( b\right) \right) ,
\end{equation*}%
and $\gamma $ is called \textit{static} if is semistatic and $\delta
_{M}\left( \gamma \left( a\right) ,\gamma \left( b\right) \right) =0$ for
any $a,b\in
\mathbb{R}
.$

Actually, the orbits contained in the Mather set $\widetilde{\mathcal{M}}%
_{c}\left( L\right) $ project onto static curves. The Mañé set $\widetilde{%
\mathcal{N}}_{c}\left( L\right) $ is the set of the points $\left(
x,v\right) \in TM$ such that the projection $\gamma \left( t\right) =\pi
\circ \varphi _{t}\left( x,v\right) $ is a semistatic curve and the Aubry
set $\widetilde{\mathcal{A}}_{c}\left( L\right) $ is the set of the points $%
\left( x,v\right) \in TM$ such that the projection $\gamma \left( t\right)
=\pi \circ \varphi _{t}\left( x,v\right) $ is a static curve.

Mañé proved that $\widetilde{\mathcal{A}}_{c}\left( L\right) $ is chain
recurrent and it is a challenging question to describe the dynamics of the
Euler-Lagrange flow restricted to $\widetilde{\mathcal{A}}_{c}\left(
L\right) $. All these properties are proven in \cite{gon7}. We introduce the
following notations: $\mathcal{A}_{c}\left( L\right) $ and $\mathcal{N}%
_{c}\left( L\right) $ to represent the projected Aubry and Mañé sets of
Lagrangian $L-c,$ respectively.

The notion of genericity in the context of Lagrangian systems is provided by
Mañé in \cite{man2}. The idea is to make special perturbations by adding a
potential: $L(x,v)+\Psi (x)$, for $\Psi \in C^{\infty }(M)$. A property is
\textit{generic} in the sense of Mañé if it is valid for every Lagrangians $%
L(x,v)+\Phi (x)$ with $\Phi $ contained in a residual subset $\mathcal{O}$.
In this sense, G. Contreras and G. Paternain, prove in \cite{gon1} (Theorem
C) that for a fixed cohomology class $c,$ the property
\begin{equation}
{\widetilde{\mathcal{M}}}_{c}\left( L\right) ={\widetilde{\mathcal{A}}}%
_{c}\left( L\right) ={\widetilde{\mathcal{N}}}_{c}\left( L\right) ,
\label{eq11}
\end{equation}%
with ${\widetilde{\mathcal{M}}}_{c}\left( L\right) $ supports a uniquely
ergodic measure is generic. Furthermore, J. Zhang proves in \cite{zhang}
that for generic Tonelli Lagrangian, there exists a residual set $\mathcal{%
G\subset }H^{1}\left( M;%
\mathbb{R}
\right) $ such that \ref{eq11} holds for any $c\in \mathcal{G}$.

In general, when we are dealing with a specific class of Lagrangians,
perturbations by adding a potential are not allowed. However the main goal
of this paper is to prove some generic properties of minimizing sets for the
family of exact magnetic Lagrangians.

In the work of Miranda, J. A. G \cite{jan1}, the perturbations of a magnetic
flow are made on the space of 2-forms on $M$ with the $C^{\infty }$ topology
that preserve the cohomology class (in particular on the exact 2-forms). He
proves a version of the Kupka-Smale Theorem for this class of flows, when $M$
is a surface. More recently, Arbieto, A. and Castro, F. \cite{arb1}
generalize for any dimension of $M$.

Let us take $\Gamma ^{\infty }\left( M\right) $ the set of smooth 1-forms on
$M\ $endowed with the metric%
\begin{equation}
d\left( \omega _{1},\omega _{2}\right) =\sum_{k\in \mathbb{N}}\frac{\arctan
\left( \left\Vert \omega _{1}-\omega _{2}\right\Vert _{k}\right) }{2^{k}},
\label{d1}
\end{equation}%
denoting by $\left\Vert \omega \right\Vert _{k}$ the $C^{k}$-norm of the
1-form $\omega ,$ that is: given a $C^{\infty \text{ }}$1-form $\omega ,$ we
associate a $C^{\infty \text{ }}$ field $X$ on $M$ such that $\omega
_{x}\left( v\right) =\left\langle X\left( x\right) ,v\right\rangle .$ We
define the $C^{k}$-norm $\left\Vert \omega \right\Vert _{k}$ of $\omega $ as
the $C^{k}$-norm of $X.$

The a first integral of the flow $\varphi _{t}$ is the \textit{energy
function }$E:TM\rightarrow
\mathbb{R}
$, defined by%
\begin{equation*}
E\left( x,v\right) =\frac{\partial L}{\partial v}\left( x,v\right) \left(
v\right) -L\left( x,v\right) .
\end{equation*}%
Recall that%
\begin{equation*}
e_{0}=-\min_{x\in M}L\left( x,0\right) =\max_{x\in M}E\left( x,0\right) .
\end{equation*}%
We have $\min \alpha \geq e_{0}$ and for any $k>e_{0},$ the energy level $%
E^{-1}\left( k\right) $ is a hypersurface of $TM.$ Let state our main result.

\begin{theorem}
\label{teorema1}For an exact magnetic Lagrangian $L$ defined on $TM$, there
exist a residual set $\mathcal{G\subset }H^{1}\left( M;%
\mathbb{R}
\right) $ and a residual $\mathcal{O\subset }\Gamma ^{\infty }\left(
M\right) $ such that%
\begin{equation*}
{\widetilde{\mathcal{M}}}_{c}\left( L+\varphi \right) ={\widetilde{\mathcal{A%
}}}_{c}\left( L+\varphi \right) ={\widetilde{\mathcal{N}}}_{c}\left(
L+\varphi \right) ,\forall c\in \mathcal{G},\forall \varphi \in \mathcal{O}%
\text{,}
\end{equation*}%
with ${\widetilde{\mathcal{M}}}_{c}\left( L+\varphi \right) $ supports a
uniquely ergodic measure.
\end{theorem}

Actually, before we will prove the following proposition which allows us to
perturb Tonelli Lagrangians by adding 1-forms.

\begin{proposition}
\label{proposicao1}Let $L$ be a Tonelli Lagrangian and $c$ be a cohomology
class with $\alpha \left( c\right) >e_{0}.$ If ${\widetilde{\mathcal{M}}}%
_{c}\left( L\right) =supp\left( \mu _{0}\right) $ supports on a uniquely
ergodic measure $\mu _{0},$ then there exists a $C^{\infty }$ 1-form $\eta $
(sufficiently close to zero) such that%
\begin{equation*}
{\widetilde{\mathcal{M}}}_{c}\left( L+\eta \right) ={\widetilde{\mathcal{A}}}%
_{c}\left( L+\eta \right) ={\widetilde{\mathcal{N}}}_{c}\left( L+\eta
\right) ={\widetilde{\mathcal{M}}}_{c}\left( L\right) .
\end{equation*}
\end{proposition}

In Section \ref{section2} we shall present a version for the family of exact
magnetic Lagrangians of a theorem stated in \cite{man3} by Mañé, proved in (%
\cite{gon5}, Theorem D).

\begin{theorem}
\label{teorema2}Let $L$ be an exact magnetic Lagrangian. Then there exists a
residual set $\mathcal{O}$ of $\Gamma ^{\infty }\left( M\right) $ such that
for every $\omega \in \mathcal{O}$ the Lagrangian $L+\omega $ has a unique
minimizing measure, uniquely ergodic $\mu _{\omega }.$ Moreover, every
periodic orbit $\Gamma $ which is the support of $\mu _{\omega },$ for some $%
\omega \in \mathcal{O}$, is hyperbolic and its stable and unstable manifolds
intersect transversally $W^{s}\left( \Gamma \right) \pitchfork W^{u}\left(
\Gamma \right) $.
\end{theorem}

\section{Perturbation of a Tonelli Lagrangian by adding a 1-form}

In this section, we will dedicate to prove Proposition \ref{proposicao1} and
Theorem \ref{teorema1}. In order to prove the results we need of the
following lemma:

\begin{lemma}
\label{lema1}Let $L:TM\rightarrow
\mathbb{R}
$ be a Tonelli Lagrangian and $c$ be a cohomology class with $\alpha \left(
c\right) >e_{0}.$ Then there exist a neighborhood $\mathcal{U\subset }$ $TM$
of ${\widetilde{\mathcal{A}}}\left( c\right) $, a $C^{\infty }$ vector field
$X:M\rightarrow TM$ and $K>0$ such that for every $\left( x,v\right) \in
\mathcal{U}$,
\begin{equation*}
\left( \frac{\partial L}{\partial v}\left( x,X\left( x\right) \right) -\frac{%
\partial L}{\partial v}\left( x,0\right) \right) \left( v\right) \geq K>0.
\end{equation*}
\end{lemma}

\begin{proof}
$\ $Since $\left( \pi |_{{\mathcal{A}}\left( c\right) }\right) ^{-1}$ is a
Lipschitz map we can consider $\xi :M\rightarrow TM$ a Lipschitz extension
of $\left( \pi |_{{\mathcal{A}}\left( c\right) }\right) ^{-1}$ to $M.$ We
define a Lipschitz map $F:TM\rightarrow
\mathbb{R}
$ by $F\left( x,v\right) =\left( \frac{\partial L}{\partial v}\left( x,\xi
\left( x\right) \right) -\frac{\partial L}{\partial v}\left( x,0\right)
\right) \left( v\right) .$ Let us prove first that $L_{0}\left( x,v\right)
=L\left( x,v\right) -\frac{\partial L}{\partial v}\left( x,0\right) \left(
v\right) -L\left( x,0\right) \geq 0.$ Indeed, observe that $L_{0}$ is a
convex superlinear function, $L_{0}\left( x,0\right) =0$ and $\frac{\partial
L_{0}}{\partial v}\left( x,0\right) =0.$ Therefore $L_{0}\left( x,\star
\right) $ has its minimum at $v=0$, hence $L_{0}\left( x,v\right) \geq 0$.
Now let us take $\left( x,v\right) \in {\widetilde{\mathcal{A}}}\left(
c\right) ,$ i.e. $v=\xi \left( x\right) .$ Then%
\begin{eqnarray*}
F\left( x,v\right) &=&\left( \frac{\partial L}{\partial v}\left( x,\xi
\left( x\right) \right) -\frac{\partial L}{\partial v}\left( x,0\right)
\right) \left( \xi \left( x\right) \right) \\
&=&\frac{\partial L}{\partial v}\left( x,\xi \left( x\right) \right) \left(
\xi \left( x\right) \right) -L\left( x,\xi \left( x\right) \right) +L\left(
x,\xi \left( x\right) \right) -\frac{\partial L}{\partial v}\left(
x,0\right) \left( \xi \left( x\right) \right) \\
&=&E\left( x,\xi \left( x\right) \right) +\left[ L\left( x,\xi \left(
x\right) \right) -\frac{\partial L}{\partial v}\left( x,0\right) \left( \xi
\left( x\right) \right) -L\left( x,0\right) \right] +L\left( x,0\right) \\
&=&\alpha \left( c\right) +L_{0}\left( x,\xi \left( x\right) \right)
+L\left( x,0\right) \geq \alpha \left( c\right) -e_{0}>0\text{,}
\end{eqnarray*}%
because $e_{0}=-\min_{x\in M}L\left( x,0\right) $.

By the continuity of $F$ we obtain a neighborhood $\mathcal{U\subset }$ $TM$
of ${\widetilde{\mathcal{A}}}\left( c\right) $ such that $F|_{{\mathcal{U}}}>%
\frac{c_{0}-e_{0}}{2}>0.$ Since ${\widetilde{\mathcal{A}}}\left( c\right)
\subset \left\{ \left( x,v\right) \in TM:\left\Vert v\right\Vert \leq
B\right\} $ for some $B>0,$ we can suppose that $\mathcal{U\subset }\left\{
\left( x,v\right) \in TM:\left\Vert v\right\Vert <D\right\} $ for some $D>0.$
By Whitney's approximation Theorem, given $\delta >0$ there exists a $%
C^{\infty }$ map $X:M\rightarrow TM$ with%
\begin{equation*}
\left\Vert X\left( x\right) -\xi \left( x\right) \right\Vert <\delta ,\text{
}\forall x\in M.\text{ }
\end{equation*}%
It follows from continuity of $L_{v}$ taking $\delta >0$ smaller if
necessary, we have%
\begin{equation*}
\left\Vert \frac{\partial L}{\partial v}\left( x,X\left( x\right) \right) -%
\frac{\partial L}{\partial v}\left( x,\xi \left( x\right) \right)
\right\Vert <\frac{\alpha \left( c\right) -e_{0}}{4D},\text{ }\forall x\in M.
\end{equation*}%
Therefore, if $\left( x,v\right) \in \mathcal{U}$,
\begin{eqnarray*}
\left\vert F\left( x,v\right) -\left( \frac{\partial L}{\partial v}\left(
x,X\left( x\right) \right) -\frac{\partial L}{\partial v}\left( x,0\right)
\right) \left( v\right) \right\vert &=&\left\vert \left( \frac{\partial L}{%
\partial v}\left( x,\xi \left( x\right) \right) -\frac{\partial L}{\partial v%
}\left( x,X\left( x\right) \right) \right) \left( v\right) \right\vert \\
&\leq &\left\Vert \frac{\partial L}{\partial v}\left( x,X\left( x\right)
\right) -\frac{\partial L}{\partial v}\left( x,\xi \left( x\right) \right)
\right\Vert \left\Vert v\right\Vert \\
&<&\frac{\alpha \left( c\right) -e_{0}}{4D}.D=\frac{\alpha \left( c\right)
-e_{0}}{4}.
\end{eqnarray*}%
and%
\begin{equation*}
\left( \frac{\partial L}{\partial v}\left( x,X\left( x\right) \right) -\frac{%
\partial L}{\partial v}\left( x,0\right) \right) v>F\left( x,v\right) -\frac{%
\alpha \left( c\right) -e_{0}}{4}>\frac{\alpha \left( c\right) -e_{0}}{4}%
\overset{def}{=}K>0.
\end{equation*}
\end{proof}

The next step is to show the upper-semicontinuity of the Mañé set for
Tonelli Lagranginas when we add a 1-form:

\begin{proposition}
\label{propa1}Let $L:TM\rightarrow
\mathbb{R}
$ be a Tonelli Lagrangian. As a set-valued function, $\left( \xi ,c\right)
\in \Gamma ^{\infty }\left( M\right) \times H^{1}\left( M;%
\mathbb{R}
\right) \longmapsto {\widetilde{\mathcal{N}}}_{c}\left( L+\xi \right) $ is
upper-semicontinuous, that is given a neighborhood $\mathcal{V}$ of ${%
\widetilde{\mathcal{N}}}_{c_{0}}\left( L+\xi _{0}\right) $ in $TM$ there
exists a neighborhood $U\times V$ of $\left( \xi _{0},c_{0}\right) $ in $%
\Gamma ^{\infty }\left( M\right) \times H^{1}\left( M;%
\mathbb{R}
\right) $ such that ${\widetilde{\mathcal{N}}}_{c}\left( L+\xi \right)
\subset $ $\mathcal{V}$ for every $\left( \xi ,c\right) \in U\times V$.
\end{proposition}

Before proving this proposition, we shall prove the Mather's $\alpha -$%
function depends continuously of $\left( \xi ,c\right) \in \Gamma ^{\infty
}\left( M\right) \times H^{1}\left( M;%
\mathbb{R}
\right) $. We will prove that it holds for any Tonelli Lagrangians.

\begin{lemma}
\label{lema2}Let $L:TM\rightarrow
\mathbb{R}
$ be a Tonelli Lagrangian. The map $\left( \xi ,\lambda \right) \in \Gamma
^{\infty }\left( M\right) \times H^{1}\left( M;%
\mathbb{R}
\right) \longmapsto \alpha _{L+\xi }\left( \lambda \right) =c\left( L+\xi
-\lambda \right) $ is continuous.
\end{lemma}

\begin{proof}
Suppose that $\left( \xi _{n},\lambda _{n}\right) \rightarrow \left( \xi
,\lambda \right) ,\alpha _{n}=\alpha _{L+\xi _{n}}\left( \lambda _{n}\right)
,$ and $\alpha =\alpha _{L+\xi }\left( \lambda \right) .$ We shall prove
that $\alpha _{n}\rightarrow \alpha .$ Let us take $\xi _{n}-\lambda
_{n}=\sigma _{n}$ and $\xi -\lambda =\sigma .$ By the duality there exist
vector fields $X_{n}$ and $X$ on $M$ such that
\begin{equation*}
\sigma _{n}\left( x\right) \left( v\right) =\left\langle X_{n}\left(
x\right) ,v\right\rangle \text{ and }\sigma \left( x\right) \left( v\right)
=\left\langle X\left( x\right) ,v\right\rangle .
\end{equation*}

Observe that the energy function $E_{n}$ of $L+\sigma _{n}$ is
\begin{equation*}
E_{n}\left( x,v\right) =\left( L\left( x,v\right) +\left\langle X_{n}\left(
x\right) ,v\right\rangle \right) _{v}v-(L+\left\langle X_{n}\left( x\right)
,v\right\rangle )=L_{v}v-L=E\left( x,v\right) ,
\end{equation*}%
for every $n,$ where $E$ is the energy function of $L+\sigma .$ It is know
that $E$ is a superlinear function. Then there exist $B>0$ such that $%
E\left( x,v\right) \geq \left\Vert v\right\Vert -B$ for every $\left(
x,v\right) \in TM.$

Observe that $-\alpha _{n}=\int_{TM}\left( L\left( x,v\right) +\sigma
_{n}\right) d\mu _{n}$ where $\mu _{n}$ is a minimizing measure $\mu _{n}$
of $L+\sigma _{n}.$ Thus $\alpha _{n}=E_{n}\left( supp\left( \mu _{n}\right)
\right) =E\left( supp\left( \mu _{n}\right) \right) .$

Since $\left( \xi _{n},\lambda _{n}\right) \rightarrow \left( \xi ,\lambda
\right) ,$ given $\varepsilon >0,$ there exists $n_{0}\in
\mathbb{N}
$ such that
\begin{equation}
n>n_{0}\Rightarrow \left\Vert X_{n}\left( x\right) -X\left( x\right)
\right\Vert <\varepsilon ,\forall x\in M.  \label{f2}
\end{equation}%
Moreover, we have that $-\alpha \leq \int_{TM}\left( L+\sigma \right) d\mu
_{n}.$ Thus, for every $n>n_{0}$, we obtain
\begin{eqnarray}
\alpha _{n}-\alpha &\leq &-\int_{TM}\left( L+\sigma _{n}\right) d\mu
_{n}+\int_{TM}\left( L+\sigma \right) d\mu _{n}  \label{f1} \\
&=&\int_{TM}\left\langle X_{n}\left( x\right) ,v\right\rangle -\left\langle
X\left( x\right) ,v\right\rangle d\mu _{n}  \notag \\
&\leq &\int_{TM}\left\Vert X_{n}\left( x\right) -X\left( x\right)
\right\Vert \left\Vert v\right\Vert d\mu _{n}<\int_{TM}\varepsilon
\left\Vert v\right\Vert d\mu _{n}  \notag \\
&\leq &\varepsilon \int_{TM}\left( E\left( x,v\right) +B\right) d\mu
_{n}=\varepsilon \alpha _{n}+\varepsilon B.  \notag
\end{eqnarray}%
Taking $\varepsilon =\frac{1}{2}\ $above, we conclude
\begin{equation*}
0<\alpha _{n}\leq \frac{\alpha +\varepsilon B}{1-\varepsilon }=2\alpha +B.
\end{equation*}%
Given $\bar{\varepsilon}>0,$ we take $\varepsilon =\min \left\{ \frac{\bar{%
\varepsilon}}{2\alpha +2B},\frac{1}{2}\right\} >0$ in (\ref{f2}). So we can
use (\ref{f1}) to obtain that there exists $n_{1}\in
\mathbb{N}
$ such that $\alpha _{n}-\alpha \leq \varepsilon \alpha _{n}+\varepsilon
B\leq \bar{\varepsilon}$ for every $n>n_{1}.$

There exists $n_{2}\in
\mathbb{N}
$ such that
\begin{equation*}
n>n_{2}\Rightarrow \left\Vert X_{n}\left( x\right) -X\left( x\right)
\right\Vert <\frac{\bar{\varepsilon}}{\alpha +B},\forall x\in M.
\end{equation*}%
Let $\mu $ be a minimizing measure of $L+\sigma .$ Then%
\begin{eqnarray*}
\alpha -\alpha _{n} &\leq &-\int_{TM}\left( L+\sigma \right) d\mu
+\int_{TM}\left( L+\sigma _{n}\right) d\mu \\
&\leq &\int_{TM}\left\Vert X_{n}\left( x\right) -X\left( x\right)
\right\Vert \left\Vert v\right\Vert d\mu \\
&\leq &\int_{TM}\frac{\bar{\varepsilon}}{\alpha +B}\left\Vert v\right\Vert
d\mu \leq \frac{\bar{\varepsilon}}{\alpha +B}\int_{TM}\left( E\left(
x,v\right) +B\right) d\mu =\bar{\varepsilon}.
\end{eqnarray*}%
Therefore, if $n>\max \left\{ n_{0},n_{1},n_{2}\right\} ,$ we have $%
\left\vert \alpha -\alpha _{n}\right\vert <\bar{\varepsilon}.$
\end{proof}

\begin{proof}
\textit{(of Proposition \ref{propa1})} Let $\left( \xi _{0},c_{0}\right) $
be a point in $\Gamma ^{\infty }\left( M\right) \times H^{1}\left( M;%
\mathbb{R}
\right) .$ Since the Mañé set is contained in the energy level ${\widetilde{%
\mathcal{N}}}_{c}\left( L\right) \subset E^{-1}\left( \alpha \left( c\right)
\right) ,$ it follows from previous lemma that there exist a neighborhood $%
U\times V$ of $\left( \xi _{0},c_{0}\right) $ in $\Gamma ^{\infty }\left(
M\right) \times H^{1}\left( M;%
\mathbb{R}
\right) $ and a compact subset $\mathcal{K\subset }TM$ such that ${%
\widetilde{\mathcal{N}}}_{c}\left( L+\xi \right) \subset $ $\mathcal{K}$ for
every $\left( c,\xi \right) \in U^{\prime }\times V^{\prime }.$

Suppose by contradiction that the Mané set ${\widetilde{\mathcal{N}}}%
_{c_{0}}\left( L+\xi _{0}\right) $ is not upper-semiconti-nuous: There
exists a neighborhood $\mathcal{V}$ of ${\widetilde{\mathcal{N}}}%
_{c_{0}}\left( L+\xi _{0}\right) $ in $TM$ such that for every neighborhood $%
U\times V$ of $\left( \xi _{0},c_{0}\right) $ in $\Gamma ^{\infty }\left(
M\right) \times H^{1}\left( M;%
\mathbb{R}
\right) $ we have ${\widetilde{\mathcal{N}}}_{c}\left( L+\xi \right)
\nsubseteqq $ $\mathcal{V}$ for some $\left( \xi ,c\right) \in U\times V$.
Then it is posssible to obtain a sequence $\left( \xi _{n},c_{n}\right) \in
U^{\prime }\times V^{\prime }$ with $\left( \xi _{n},c_{n}\right)
\rightarrow \left( \xi _{0},c_{0}\right) $ in $\Gamma ^{\infty }\left(
M\right) \times H^{1}\left( M;%
\mathbb{R}
\right) $ and $\left( x_{n},v_{n}\right) \in {\widetilde{\mathcal{N}}}%
_{c_{n}}\left( L+\xi _{n}\right) \setminus \mathcal{V}$. Since ${\widetilde{%
\mathcal{N}}}_{c_{n}}\left( L+\xi _{n}\right) \setminus \mathcal{V}$ is
contained in the compact set $\mathcal{K}$ for every $n,$ we can suppose the
convergence $\left( x_{n},v_{n}\right) \rightarrow \left( x_{0},v_{0}\right)
\notin {\widetilde{\mathcal{N}}}_{c_{0}}\left( L+\xi _{0}\right) .\ $

We shall prove that the Euler Lagrange solution $\left( \gamma _{n}\left(
t\right) ,\dot{\gamma}_{n}\left( t\right) \right) =\varphi _{t}^{L+\xi
_{n}-c_{n}}\left( x_{n},v_{n}\right) $ converges on the compacts of the form
$\left[ 0,T\right] $:%
\begin{equation*}
\varphi _{t}^{L+\xi _{n}-c_{n}}\left( x_{n},v_{n}\right) \rightarrow \varphi
_{t}^{L+\xi _{0}-c_{0}}\left( x_{0},v_{0}\right) .
\end{equation*}%
Indeed, let $K=\sup\limits_{\left( x,v\right) \in \mathcal{K}}L\left(
x,v\right) .$ So%
\begin{equation*}
\int_{0}^{T}L\left( \gamma _{n}\left( t\right) ,\dot{\gamma}_{n}\left(
t\right) \right) dt\leq KT,
\end{equation*}%
Since each $\gamma _{n}$ is a $C^{k}$-curve and the actions of $L$ on the
curves $\gamma _{n}|_{\left[ 0,T\right] }$ are bounded by $KT$, we have that
the set $\left\{ \gamma _{n}\right\} $ is compact in the $C^{0}$-topology.
Actually, this set is compact in the $C^{1}$-topology because we have $%
\left\Vert \dot{\gamma}_{n}\right\Vert $ bounded and $L_{vv}$ positive
definite. Moreover, if $\gamma _{0}$ is a limit point of $\left\{ \gamma
_{n}|_{\left[ 0,T\right] }\right\} $, so $\gamma _{0}$ is a Tonelli
minimizing for the Lagrangian $L+\xi _{0}-c_{0}$. Thus if $y_{0}=\gamma
_{0}\left( T\right) $ we have
\begin{eqnarray}
\Phi _{c_{0}}^{L+\xi _{0}}\left( x_{0},y_{0}\right) &\leq &A_{c_{0}}^{L+\xi
_{0}}\left( \gamma _{0}|_{\left[ 0,T\right] }\right) +\alpha _{L+\xi
_{0}}\left( c_{0}\right) T  \label{eq1} \\
&=&\lim_{n}\left[ A_{c_{n}}^{L+\xi _{n}}\left( \gamma _{n}|_{\left[ 0,T%
\right] }\right) +\alpha _{L+\xi _{n}}\left( c_{n}\right) T\right]  \notag \\
&=&\lim_{n}\Phi _{c_{n}}^{L+\xi _{n}}\left( x_{n},\gamma _{n}\left( T\right)
\right) ,  \notag
\end{eqnarray}%
Write $\Delta =\lim \Phi _{c_{n}}^{L+\xi _{n}}\left( x_{n},\gamma _{n}\left(
T\right) \right) .$ If we prove that $\Delta =\Phi _{c_{0}}^{L+\xi
_{0}}\left( x_{0},y_{0}\right) ,$ then we have equality of (\ref{eq1}):
\begin{equation*}
A_{c_{0}}^{L+\xi _{0}}\left( \gamma _{0}|_{\left[ 0,T\right] }\right)
+\alpha _{L+\xi _{0}}\left( c_{0}\right) T=\Phi _{c_{0}}^{L+\xi _{0}}\left(
x_{0},y_{0}\right) ,
\end{equation*}%
that is $\left( x_{0},v_{0}\right) \in {\widetilde{\mathcal{N}}}%
_{c_{0}}\left( L+\xi _{0}\right) $ and we obtain a contradiction. If $\Phi
_{c_{0}}^{L+\xi _{0}}\left( x_{0},y_{0}\right) <\Delta -\varepsilon $ for
some $\varepsilon >0,$ then by definition of Mañé's potential, there exists
a curve $\sigma :\left[ 0,S\right] \rightarrow M$ with $\sigma \left(
0\right) =x_{0}$ and $\sigma \left( S\right) =y_{0}$ such that%
\begin{equation*}
\Phi _{c_{0}}^{L+\xi _{0}}\left( x_{0},y_{0}\right) \leq A_{c_{0}}^{L+\xi
_{0}}\left( \sigma |_{\left[ 0,S\right] }\right) +\alpha _{L+\xi _{0}}\left(
c_{0}\right) S<\Delta -\varepsilon .
\end{equation*}%
By triangular inequality property:%
\begin{eqnarray}
\Phi _{c_{n}}^{L+\xi _{n}}\left( x_{n},\gamma _{n}\left( T\right) \right)
&\leq &\Phi _{c_{n}}^{L+\xi _{n}}\left( x_{n},x_{0}\right) +\Phi
_{c_{n}}^{L+\xi _{n}}\left( x_{0},y_{0}\right) +\Phi _{c_{n}}^{L+\xi
_{n}}\left( \gamma _{0},\gamma _{n}\left( T\right) \right)  \label{eq5} \\
&\leq &A_{c_{n}}^{L+\xi _{n}}\left( \sigma |_{\left[ 0,S\right] }\right)
+\alpha _{L+\xi _{n}}\left( c_{n}\right) S+\Phi _{c_{n}}^{L+\xi _{n}}\left(
x_{n},x_{0}\right) +\Phi _{c_{n}}^{L+\xi _{n}}\left( \gamma _{0},\gamma
_{n}\left( T\right) \right) .  \notag
\end{eqnarray}%
Given $p,q\in M$ let us take $\eta $ a geodesic with speed of norm $1$ from $%
p$ to $q$ and $d=d_{M}\left( p,q\right) .$ Hence
\begin{eqnarray*}
\Phi _{c_{n}}^{L+\xi _{n}}\left( p,q\right) &\leq &A_{c_{n}}^{L+\xi
_{n}}\left( \eta |_{\left[ 0,d\right] }\right) +\alpha _{L+\xi _{n}}\left(
c_{n}\right) d \\
&=&\int_{0}^{d}\left[ L\left( \eta ,\dot{\eta}\right) +\left( \xi
_{n}-c_{n}\right) \left( \dot{\eta}\right) +\alpha _{L+\xi _{n}}\left(
c_{n}\right) \right] dt \\
&\leq &\left( \max_{\left\Vert v\right\Vert =1}\left\vert L\left( x,v\right)
\right\vert +\max_{\left\Vert v\right\Vert =1}\left\vert \left\langle \left(
\xi _{n}+c_{n}\right) \left( x\right) ,v\right\rangle \right\vert +\alpha
_{L+\xi _{n}}\left( c_{n}\right) \right) d.
\end{eqnarray*}%
By the continuity of critical value proved in Lemma \ref{lema2}, we obtain
that exists $K>0$ such that for $n$ suficiently large we have $\Phi
_{c_{n}}^{L+\xi _{n}}\left( p,q\right) \leq Kd_{M}\left( p,q\right) .$
Therefore, letting $n\rightarrow \infty $ we get $\Phi _{c_{n}}^{L+\xi
_{n}}\left( x_{n},x_{0}\right) \rightarrow 0,$ $\Phi _{c_{n}}^{L+\xi
_{n}}\left( \gamma _{0},\gamma _{n}\left( T\right) \right) \rightarrow 0$
and, by inequality (\ref{eq5}), we obtain a contradiction:
\begin{equation*}
\Delta =\lim \Phi _{c_{n}}^{L+\xi _{n}}\left( x_{n},\gamma _{n}\left(
T\right) \right) \leq A_{c_{0}}^{L+\xi _{0}}\left( \sigma |_{\left[ 0,S%
\right] }\right) +\alpha _{L+\xi _{0}}\left( c_{0}\right) S<\Delta
-\varepsilon .
\end{equation*}
\end{proof}

Now it is possible to conclude the proof of Proposition \ref{proposicao1}
stated in Introduction:

\begin{proof}
\textit{(of Proposition \ref{proposicao1}) }Let us take $\mathcal{U\subset }%
TM$ the neighborhood of ${\widetilde{\mathcal{A}}}_{c}\left( L\right) $
given by Lemma \ref{lema1} and $B=\pi \left( \mathcal{U}\right) \subset M$
(open subset of $M).$ Given $\varepsilon >0,$ let $\lambda :$ $M\rightarrow
\mathbb{R}
$ be a $C^{\infty }$ function given by%
\begin{equation}
\lambda \left( x\right) =\left\{
\begin{array}{l}
0,\text{ on }{\mathcal{M}}_{c}\left( L\right) \\
g\left( x\right) \text{ on }B\setminus {\mathcal{M}}_{c}\left( L\right) \\
0\text{ on }M\setminus B%
\end{array}%
\right. ,  \label{eq6}
\end{equation}%
fixed a function $g$ with $0<g\left( x\right) <\varepsilon .$ Let us take
the $C^{\infty }$ 1-form given by $\eta _{\varepsilon }\left( x\right)
\left( v\right) =\lambda \left( x\right) \left( \frac{\partial L}{\partial v}%
\left( x,X\left( x\right) \right) -\frac{\partial L}{\partial v}\left(
x,0\right) \right) \left( v\right) $ where the field $X$ is given by Lemma %
\ref{lema1}. Since $\mathfrak{M}_{c}\left( L\right) =\left\{ \mu
_{0}\right\} ,$ it follows from Lemma 5.3 in \cite{gon1} that ${\widetilde{%
\mathcal{A}}}_{c}\left( L\right) ={\widetilde{\mathcal{N}}}_{c}\left(
L\right) .$ Moreover since Mañé set is upper-semicontinuous with respect to
1-forms (Proposition \ref{propa1}), for $\varepsilon $ sufficiently small,
taking $\eta =\eta _{\varepsilon }$ we obtain ${\widetilde{\mathcal{N}}}%
_{c}\left( L+\eta _{\varepsilon }\right) \subset \mathcal{U}$. Hence
\begin{equation}
{\widetilde{\mathcal{M}}}_{c}\left( L+\eta \right) \subset {\widetilde{%
\mathcal{A}}}_{c}\left( L+\eta \right) \subset {\widetilde{\mathcal{N}}}%
_{c}\left( L+\eta \right) \subset \mathcal{U}.  \label{eq4}
\end{equation}%
This means that for every $\mu \in \mathfrak{M}_{c}\left( L+\eta \right) $
we have $supp\left( \mu \right) \subset \mathcal{U}$ and by Lemma \ref{lema1}
we obtain%
\begin{equation*}
\int_{TM}\eta d\mu =\int_{TM}\eta \left( x\right) vd\mu \geq 0.
\end{equation*}%
Let us to show that%
\begin{equation*}
{\widetilde{\mathcal{M}}}_{c}\left( L+\eta \right) ={\widetilde{\mathcal{M}}}%
_{c}\left( L\right) .
\end{equation*}%
Indeed, since $\eta |_{{\widetilde{\mathcal{M}}}_{c}\left( L\right) }\equiv
0,$ we have
\begin{equation*}
A_{c}^{L+\eta }\left( \mu _{0}\right) =A_{c}^{L}\left( \mu _{0}\right) \leq
A_{c}^{L}\left( \mu \right) \leq A_{c}^{L+\eta }\left( \mu \right) ,\forall
\mu \in \mathfrak{M}_{c}\left( L+\eta \right) ,
\end{equation*}%
That is%
\begin{equation*}
\mu _{0}\in \mathfrak{M}_{c}\left( L+\eta \right) .
\end{equation*}%
On the other hand, if $\delta \in \mathfrak{M}_{c}\left( L+\eta \right) ,$%
\begin{equation*}
A_{c}^{L}\left( \delta \right) \leq A_{c}^{L+\eta }\left( \delta \right)
=A_{c}^{L+\eta }\left( \mu _{0}\right) =A_{c}^{L}\left( \mu _{0}\right) ,
\end{equation*}%
Therefore $\delta \in \mathfrak{M}_{c}\left( L\right) $ and we conclude that%
\begin{equation*}
{\widetilde{\mathcal{M}}}_{c}\left( L+\eta \right) =supp\left( \mu
_{0}\right) .
\end{equation*}%
Let us suppose that ${\widetilde{\mathcal{A}}}_{c}\left( L+\eta \right) \neq
{\widetilde{\mathcal{M}}}_{c}\left( L+\eta \right) .$ Recall that since the
Graph Property holds for ${\widetilde{\mathcal{A}}}_{c}\left( L+\eta \right)
,$ there exists
\begin{equation*}
x\in {\mathcal{A}}_{c}\left( L+\eta \right) \setminus {\mathcal{M}}%
_{c}\left( L+\eta \right) .
\end{equation*}%
Let $\gamma :%
\mathbb{R}
\rightarrow M$ be the minimizing curve for $L+\eta -c$ with $\gamma \left(
0\right) =x.$ It follows from \ref{eq4} that $\left( \gamma ,\dot{\gamma}%
\right) \left(
\mathbb{R}
\right) \subset \mathcal{U}.$ Since $\mu _{0}$ is ergodic, almost every
point has a dense orbit on $supp\left( \mu _{0}\right) .$ Let $z\in $ $%
supp\left( \mu _{0}\right) $ be such that it has a dense orbit $\left(
\sigma ,\dot{\sigma}\right) $ on $supp\left( \mu _{0}\right) $. Then given $%
u,w\in \pi \left( supp\left( \mu _{0}\right) \right) ,$ we can take $%
t_{n}>s_{n}>0$ such that $u=\lim_{n}\sigma \left( s_{n}\right) $ and $%
w=\lim_{n}\sigma \left( t_{n}\right) .$ Observe that the critical values $%
\alpha _{L}\left( c\right) =\alpha _{L+\eta }\left( c\right) =\alpha $ are
the same. Moreover as we mentioned before $\eta |_{supp\left( \mu
_{0}\right) }\equiv 0.$ Hence
\begin{eqnarray*}
\Phi _{c}^{L+\eta }\left( u,w\right) &=&\lim_{n}\Phi _{c}^{L+\eta }\left(
\sigma \left( s_{n}\right) ,\sigma \left( t_{n}\right) \right)
=\lim_{n}A_{c}^{L+\eta +\alpha }\left( \sigma |_{\left[ s_{n},t_{n}\right]
}\right) \\
&=&\lim_{n}A_{c}^{L+\alpha }\left( \sigma |_{\left[ s_{n},t_{n}\right]
}\right) =\lim_{n}\Phi _{c}^{L}\left( \sigma \left( s_{n}\right) ,\sigma
\left( t_{n}\right) \right) \\
&=&\Phi _{c}^{L}\left( u,w\right) .
\end{eqnarray*}

It is known that for every $\left( y,w\right) \in {\widetilde{\mathcal{N}}}%
_{c}\left( L+\eta \right) ,$ the $\omega $ and $\alpha $-limit sets of $%
\left( y,w\right) $ are contained in $supp\left( \mu _{0}\right) $ (because $%
{\widetilde{\mathcal{M}}}_{c}\left( L+\eta \right) =supp\left( \mu
_{0}\right) )$. Hence $\omega $ and $\alpha $-limits of $x$ are contained in
$\pi \left( supp\left( \mu _{0}\right) \right) .$ Now let $%
t_{n},s_{n}\rightarrow \infty $ such that $u_{0}=\lim_{n}\gamma \left(
t_{n}\right) $ and $w_{0}=\lim_{n}\gamma \left( -s_{n}\right) .$ We observe
that $x\notin {\mathcal{M}}_{c}\left( L\right) ={\mathcal{M}}_{c}\left(
L+\eta \right) $ then for $n_{0}$ sufficiently big, $\gamma |_{\left[
t_{n_{0}},t_{n}\right] }\subset B\setminus {\mathcal{M}}_{c}\left( L\right) $
for every $t_{n}>t_{n_{0}}.$ By the definition of $\lambda ,$ there exists $%
a>0$ such that $\int_{-s_{n}}^{t_{n}}\lambda \left( \gamma \left( t\right)
\right) dt>a.$ Therefore by Lemma \ref{lema1} we have
\begin{equation*}
\int_{-s_{n}}^{t_{n}}\eta \left( \gamma \left( t\right) \right) \dot{\gamma}%
\left( t\right) dt\geq \int_{-s_{n}}^{t_{n}}\lambda \left( \gamma \left(
t\right) \right) Kdt>aK.
\end{equation*}%
Hence%
\begin{eqnarray*}
0 &=&\Phi _{c}^{L+\eta }\left( \gamma \left( -s_{n}\right) ,\gamma \left(
t_{n}\right) \right) +\Phi _{c}^{L+\eta }\left( \gamma \left( t_{n}\right)
,\gamma \left( -s_{n}\right) \right) \\
&=&A_{c}^{L+\eta +\alpha }\left( \gamma |_{\left[ -s_{n},t_{n}\right]
}\right) +\Phi _{c}^{L+\eta }\left( \gamma \left( t_{n}\right) ,\gamma
\left( -s_{n}\right) \right) \\
&>&A_{c}^{L+\alpha }\left( \gamma |_{\left[ -s_{n},t_{n}\right] }\right)
+aK+\Phi _{c}^{L+\eta }\left( \gamma \left( t_{n}\right) ,\gamma \left(
-s_{n}\right) \right) \\
&\geq &\Phi _{c}^{L}\left( \gamma \left( -s_{n}\right) ,\gamma \left(
t_{n}\right) \right) +\Phi _{c}^{L+\eta }\left( \gamma \left( t_{n}\right)
,\gamma \left( -s_{n}\right) \right) +aK.
\end{eqnarray*}%
Taking limit as $n\rightarrow \infty $ we obtain%
\begin{eqnarray*}
0 &\geq &\Phi _{c}^{L}\left( w_{0},u_{0}\right) +\Phi _{c}^{L+\eta }\left(
u_{0},w_{0}\right) +aK \\
&=&\Phi _{c}^{L}\left( w_{0},u_{0}\right) +\Phi _{c}^{L}\left(
u_{0},w_{0}\right) +aK\geq aK.
\end{eqnarray*}%
This contradiction implies that
\begin{equation*}
{\widetilde{\mathcal{A}}}_{c}\left( L+\eta \right) ={\widetilde{\mathcal{M}}}%
_{c}\left( L+\eta \right) ={\widetilde{\mathcal{M}}}_{c}\left( L\right) .
\end{equation*}

Since $\#\mathfrak{M}_{c}\left( L+\eta \right) =1$ it follows from Lemma 5.3
in \cite{gon1} that ${\widetilde{\mathcal{A}}}_{c}\left( L+\eta \right) ={%
\widetilde{\mathcal{N}}}_{c}\left( L+\eta \right) .$ Then we conclude the
proof.
\end{proof}

\subsection{The exact magnetic Lagrangian case}

Now let $L$ be an exact magnetic Lagrangian. That is%
\begin{equation*}
L\left( x,v\right) =\frac{\left\Vert v\right\Vert ^{2}}{2}+\xi _{x}\left(
v\right)
\end{equation*}%
for some non-closed 1-form $\xi _{x}$. We will prove that ${\widetilde{%
\mathcal{M}}}_{c}\left( L\right) $ supports on a uniquely ergodic measure
for a residual set of $H^{1}\left( M;%
\mathbb{R}
\right) .$ In order for this, we need of following conclusion derived from
Theorem 1.1 in \cite{car3}:

\begin{theorem}
\label{teo2}Let $L$ be an exact magnetic Lagrangian. Given a cohomology
class $c,$ there exists a residual subset $\mathcal{O}_{c}$ of $\Gamma
^{\infty }\left( M\right) $ such that for any $\omega \in \mathcal{O}_{c},$ $%
{\widetilde{\mathcal{M}}}_{c}\left( L+\omega \right) $ supports on a
uniquely ergodic measure$.$
\end{theorem}

Since the subset $\Lambda \subset \Gamma ^{\infty }\left( M\right) $ of
non-closed 1-forms is open and dense in $\Gamma ^{\infty }\left( M\right) ,$
we can consider the residual $\mathcal{O}_{c},$ intercepting with $\Lambda $
if necessary, such that its elements are non-closed 1-forms and such that $%
\xi _{x}+\omega $ are non-closed 1-forms. This means that $L+\omega $ is
also an exact magnetic Lagrangian. In this case, the magnetic field of the
perturbed Lagrangian changes the Lorentz force.

Observe that the 1-form obtained in Proposition \textit{\ref{proposicao1}}
is given by
\begin{equation}
\eta _{x}\left( v\right) =\lambda \left( x\right) \left( \frac{\partial L}{%
\partial v}\left( x,X\left( x\right) \right) -\frac{\partial L}{\partial v}%
\left( x,0\right) \right) \left( v\right) ,  \label{eq7}
\end{equation}%
where the $C^{\infty }$ function $\lambda :$ $M\rightarrow
\mathbb{R}
$ is given in \ref{eq6} and the field $X$ is given by Lemma \ref{lema1}.
Then for the exact magnetic Lagrangian case, we have $\eta _{x}\left(
v\right) =\left\langle \lambda \left( x\right) X\left( x\right)
,v\right\rangle .$

With these notations we obtain the following corollary:

\begin{corollary}
\label{coro1}Let $L$ be an exact magnetic Lagrangian. Given a cohomology
class $c\in $ $H^{1}\left( M;%
\mathbb{R}
\right) $ and a 1-form $\omega \in \mathcal{O}_{c}$, there exists a $%
C^{\infty }$ 1-form $\eta $ (sufficiently close to zero) such that%
\begin{equation*}
{\widetilde{\mathcal{M}}}_{c}\left( L+\omega +\eta \right) ={\widetilde{%
\mathcal{A}}}_{c}\left( L+\omega +\eta \right) ={\widetilde{\mathcal{N}}}%
_{c}\left( L+\omega +\eta \right) ={\widetilde{\mathcal{M}}}_{c}\left(
L+\omega \right) ,
\end{equation*}%
with ${\widetilde{\mathcal{M}}}_{c}\left( L+\omega \right) $ supports on a
uniquely ergodic measure.
\end{corollary}

\begin{proof}
It follows from \cite{pat1}, Corollary 5.1 that $\alpha \left( c\right)
>e_{0}$ for every $c\in $ $H^{1}\left( M;%
\mathbb{R}
\right) .$ Then the proof follows directly from Proposition \textit{\ref%
{proposicao1}} and remarks above.
\end{proof}

In order to state some direct consequences from Proposition \textit{\ref%
{proposicao1}}, let us take $\zeta =\left\{ c_{n}\right\} _{n=1}^{\infty }$
a dense sequence in $H^{1}\left( M;%
\mathbb{R}
\right) .$

\begin{corollary}
\label{coro2}Let $L$ be an exact magnetic Lagrangian. Then there exists a
residual subset $\mathcal{O}^{\prime }$ of $\Gamma ^{\infty }\left( M\right)
$ such that for each $\omega \in \mathcal{O}^{\prime }$ and each $c_{n}\in
\zeta ,$ there exists a 1-form $\eta _{\omega ,n}$ (sufficiently close to $%
0) $ such that
\begin{equation*}
{\widetilde{\mathcal{M}}}_{c_{n}}\left( L+\omega +\eta _{\omega ,n}\right) ={%
\widetilde{\mathcal{A}}}_{c_{n}}\left( L+\omega +\eta _{\omega ,n}\right) ={%
\widetilde{\mathcal{N}}}_{c_{n}}\left( L+\omega +\eta _{\omega ,n}\right) ,
\end{equation*}%
with ${\widetilde{\mathcal{M}}}_{c_{n}}\left( L+\omega +\eta _{\omega
,n}\right) $ supports on a uniquely ergodic measure.
\end{corollary}

\begin{proof}
It follows from Theorem \ref{teo2} that there exists a residual subset $%
\mathcal{O}_{c_{n}}$ of $\Gamma ^{\infty }\left( M\right) $ such that for
any $\omega \in \mathcal{O}_{c_{n}},$ ${\widetilde{\mathcal{M}}}%
_{c_{n}}\left( L+\omega \right) $ supports on a uniquely ergodic measure.
Let $\mathcal{O}^{\prime }$ be the residual subset $\mathcal{O}^{\prime }%
\mathcal{=\cap O}_{c_{n}}$ of $\Gamma ^{\infty }\left( M\right) .$ Now, for
each $\omega \in \mathcal{O}^{\prime }$, by taking $L+\omega $ and $c=c_{n}$
in Proposition \textit{\ref{proposicao1}}, we conclude that there exists $%
\eta _{\omega ,n}$ sufficiently close to $0$ that satisfies the statement.
\end{proof}

\begin{corollary}
\label{coro3}Let $L$ be an exact magnetic Lagrangian. Then for each $%
c_{n}\in \zeta $ there exists a dense set $\mathcal{O}_{n}$ of $\Gamma
^{\infty }\left( M\right) $ such that
\begin{equation*}
{\widetilde{\mathcal{M}}}_{c_{n}}\left( L+\varphi _{n}\right) ={\widetilde{%
\mathcal{A}}}_{c_{n}}\left( L+\varphi _{n}\right) ={\widetilde{\mathcal{N}}}%
_{c_{n}}\left( L+\varphi _{n}\right) ,
\end{equation*}%
and ${\widetilde{\mathcal{M}}}_{c_{n}}\left( L+\varphi _{n}\right) $
supports on a uniquely ergodic measure, for every $\varphi _{n}\in \mathcal{O%
}_{n}.$
\end{corollary}

\begin{proof}
Let $\mathcal{O}^{\prime }$ be the residual subset obtained from Corollary %
\ref{coro2}. We can vary $\omega \in \mathcal{O}^{\prime }$(which is dense
in $\Gamma ^{\infty }\left( M\right) )$ in order to obtain the following
dense set
\begin{equation*}
\mathcal{O}_{n}=\left\{ \varphi _{n}=\omega +\eta _{\omega ,n}:\omega \in
\mathcal{O}^{\prime }\right\} ,
\end{equation*}
for each $n\in
\mathbb{N}
$. This proves the statement.
\end{proof}

Finally, we can prove Theorem \ref{teorema1}:

\begin{proof}
\textit{(of Theorem \ref{teorema1}) }Fix $c_{n}\in \zeta $ and $\sigma
=\varphi _{n}\in \mathcal{O}_{n}$ given by Corollary \ref{coro3}. Consider
the following neighborhood of ${\widetilde{\mathcal{M}}}_{c_{n}}\left(
L+\sigma \right) $ in $TM:$
\begin{equation*}
\mathcal{V}_{n,r}\left( \sigma \right) \mathcal{=}\left\{ P\in
TM:d_{TM}\left( P,{\widetilde{\mathcal{M}}}_{c_{n}}\left( L+\sigma \right)
\right) <\frac{1}{r}\right\} ,
\end{equation*}%
Given $\varepsilon >0,$ let $\left\{ z_{1},...,z_{N}\right\} \subset {%
\widetilde{\mathcal{M}}}_{c_{n}}\left( L+\sigma \right) =supp\left( \mu
_{n}\right) \ $be such that $supp\left( \mu _{n}\right) \subset
\bigcup\nolimits_{i=1}^{N}B\left( z_{i},\frac{1}{r}\right) ,$ where $B\left(
z_{i},\frac{1}{r}\right) \subset \mathcal{V}_{n,r}\left( \sigma \right) $ is
the open ball of center $z_{i}$ and radius $\frac{1}{r}$ in $TM.$ There
exists a open $A_{n,r}\left( \sigma \right) $ of $\left( \sigma
,c_{n}\right) $ in $\Gamma ^{\infty }\left( M\right) \times H^{1}\left( M;%
\mathbb{R}
\right) $ such that for every $\left( \xi ,c\right) \in A_{n,r}\left( \sigma
\right) $ we have%
\begin{equation}
\sup_{P\in {\widetilde{\mathcal{M}}}_{c_{n}}\left( L+\sigma \right)
}d_{TM}\left( P,{\widetilde{\mathcal{M}}}_{c}\left( L+\xi \right) \right) <%
\frac{1}{r}.  \label{eq2}
\end{equation}%
Indeed, otherwise, we obtain sequences $c_{n}^{k}\rightarrow c_{n},\xi
_{k}\rightarrow \sigma $ (as $k\rightarrow \infty $) and $P_{k}\in {%
\widetilde{\mathcal{M}}}_{c_{n}}\left( L+\sigma \right) =supp\left( \mu
_{n}\right) $ such that $d_{TM}\left( P_{k},{\widetilde{\mathcal{M}}}%
_{c_{n}^{k}}\left( L+\xi _{k}\right) \right) \geq \frac{1}{r}.$ We consider
a sequence of minimizing measures $\mu _{n}^{k},supp\left( \mu
_{n}^{k}\right) \subset {\widetilde{\mathcal{M}}}_{c}\left( L+\xi \right) .$
The continuity of $\alpha _{L+\xi }\left( c\right) $ implies that $\mu
_{n}^{k}\rightarrow \mu _{n}$ (as $k\rightarrow \infty $) weakly*. Hence if $%
g_{i}:TM\rightarrow
\mathbb{R}
$ is a positive continuous function with
\begin{equation*}
B\left( z_{i},\frac{1}{r}\right) \cap supp\left( \mu _{n}\right) \subset
supp\left( g_{i}\right) \subset B\left( z_{i},\frac{1}{r}\right) ,
\end{equation*}%
we have $\int g_{i}d\mu _{n}>0$ for every $i=1,...,N.$ Then $\int g_{i}d\mu
_{n}^{k}>0,$ for $k$ sufficiently big and for any $i=1,...,N$, that implies $%
B\left( z_{i},\frac{1}{r}\right) \cap supp\left( \mu _{n}^{k}\right) \neq
\emptyset .$ Since $P_{k}\in B\left( z_{i},\frac{1}{r}\right) $ for some $%
i=1,...,N$ we obtain $d_{TM}\left( P_{k},supp\left( \mu _{n}^{k}\right)
\right) <\frac{1}{r}.$ This is an absurd.

As the map $\left( \xi ,c\right) \in \Gamma ^{\infty }\left( M\right) \times
H^{1}\left( M;%
\mathbb{R}
\right) \longmapsto {\widetilde{\mathcal{N}}}_{c}\left( L+\xi \right) $ is
upper-semicontinuous and ${\widetilde{\mathcal{M}}}_{c_{n}}\left( L+\sigma
\right) ={\widetilde{\mathcal{N}}}_{c_{n}}\left( L+\sigma \right) $ for
every $\sigma \in \mathcal{O}_{n},$ we can consider the neighborhood $%
A_{n,r}\left( \sigma \right) $ of $\left( \sigma ,c_{n}\right) $ in $\Gamma
^{\infty }\left( M\right) \times H^{1}\left( M;%
\mathbb{R}
\right) $ such that ${\widetilde{\mathcal{N}}}_{c}\left( L+\xi \right)
\subset $ $\mathcal{V}_{n,r}\left( \sigma \right) $ for every $\left( \xi
,c\right) \in A_{n,r}\left( \sigma \right) $ and for each $\sigma \in
\mathcal{O}_{n}.$ Let us take the subset open $B_{n,r}=\bigcup\limits_{%
\sigma \in \mathcal{O}_{n}}A_{n,r}\left( \sigma \right) \subset \Gamma
^{\infty }\left( M\right) \times H^{1}\left( M;%
\mathbb{R}
\right) ,U_{n,r}=\pi _{1}\left( B_{n,r}\right) $ and $V_{n,r}=\pi _{2}\left(
B_{n,r}\right) .$ Hence $U_{n,r}\supset \mathcal{O}_{n}$ and%
\begin{equation*}
{\widetilde{\mathcal{M}}}_{c}\left( L+\xi \right) \subset {\widetilde{%
\mathcal{A}}}_{c}\left( L+\xi \right) \subset {\widetilde{\mathcal{N}}}%
_{c}\left( L+\xi \right) \subset \mathcal{V}_{n,r}\left( \sigma \right)
,\forall \left( \xi ,c\right) \in U_{n,r}\times V_{n,r}.
\end{equation*}%
Let us take $\mathcal{O}^{\prime \prime }=\bigcap\limits_{n,r}U_{n,r},$ $%
\mathcal{O}=\mathcal{O}^{\prime }\cap \mathcal{O}^{\prime \prime }$ (given
by Corollary \ref{coro2}) and $\mathcal{G=}\bigcap\limits_{r}\left(
\bigcup\limits_{n}V_{n,r}\right) .$ Observe that $\bigcup\limits_{n}V_{n,r}$
is a open and dense subset of $H^{1}\left( M;%
\mathbb{R}
\right) .$

Now let us to show that
\begin{equation*}
{\widetilde{\mathcal{M}}}_{c}\left( L+\xi \right) ={\widetilde{\mathcal{A}}}%
_{c}\left( L+\xi \right) ={\widetilde{\mathcal{N}}}_{c}\left( L+\xi \right)
,\forall \left( \xi ,c\right) \in \mathcal{O}\times \mathcal{G}\text{.}
\end{equation*}%
Indeed, let us take $R\in {\widetilde{\mathcal{N}}}_{c}\left( L+\xi \right) $
and an integer $k>0.$ As $c\in \mathcal{G}$ we can find $m\in
\mathbb{N}
$ such that $c\in V_{m,2k}$ (Neighborhood of $c_{m}$)$.$ Since $\xi \in
\mathcal{O}$ we have $\xi \in U_{m,2k}.$ Hence $\left( \xi ,c\right) \in
B_{m,2k}$ and there exists $\sigma _{m}\in \mathcal{O}_{m}$ such that $%
\left( \xi ,c\right) \in A_{m,2k}\left( \sigma _{m}\right) ,$that is%
\begin{equation}
{\widetilde{\mathcal{M}}}_{c}\left( L+\xi \right) \subset {\widetilde{%
\mathcal{A}}}_{c}\left( L+\xi \right) \subset {\widetilde{\mathcal{N}}}%
_{c}\left( L+\xi \right) \subset \mathcal{V}_{m,2k}\left( \sigma _{m}\right)
.  \label{eq3}
\end{equation}%
Let $Q\in {\widetilde{\mathcal{M}}}_{c_{m}}\left( L+\sigma _{m}\right) $ be
the minimum point:%
\begin{equation*}
d_{TM}\left( R,{\widetilde{\mathcal{M}}}_{c_{m}}\left( L+\sigma _{m}\right)
\right) =d_{TM}\left( R,{Q}\right) ,
\end{equation*}%
and $S\in {\widetilde{\mathcal{M}}}_{c}\left( L+\xi \right) $ the minimum
point:%
\begin{equation*}
d_{TM}\left( Q,{\widetilde{\mathcal{M}}}_{c}\left( L+\xi \right) \right)
=d_{TM}\left( Q,{S}\right) .
\end{equation*}%
It follows from \ref{eq2} and \ref{eq3} that%
\begin{eqnarray*}
d_{TM}\left( R,{\widetilde{\mathcal{M}}}_{c}\left( L+\xi \right) \right)
&\leq &.d_{TM}\left( R,S\right) \leq d_{TM}\left( R,Q\right) +d_{TM}\left(
Q,S\right) \\
&=&d_{TM}\left( R,{\widetilde{\mathcal{M}}}_{c_{m}}\left( L+\sigma
_{m}\right) \right) +d_{TM}\left( Q,{\widetilde{\mathcal{M}}}_{c}\left(
L+\xi \right) \right) \\
&<&\frac{1}{2k}+\frac{1}{2k}=\frac{1}{k}.
\end{eqnarray*}%
Since this holds for any $k>0$ and ${\widetilde{\mathcal{M}}}_{c}\left(
L+\xi \right) $ is compact, we conclude that $R\in {\widetilde{\mathcal{M}}}%
_{c}\left( L+\xi \right) .$
\end{proof}

\section{Hyperbolic periodic orbit for a perturbed exact magnetic Lagrangian
\label{section2}}

The main goal of this section is to prove Theorem \ref{teorema2}. Before we
must prove the following proposition that holds for Tonelli Lagrangians. The
idea of the proof is to use the results obtained by G. Contreras and R.
Iturriaga in \cite{gon5} on the index forms. Let $\Omega _{T}$ be the set of
continuous piecewise $C^{2}$ vectorfields $\xi $ along a curve $\gamma |_{%
\left[ 0,T\right] }.$ The \textit{index form }on $\Omega _{T}$ is defined by%
\begin{equation}
I\left( \xi ,\zeta \right) =\int_{0}^{T}\left( L_{vv}\left( \dot{\xi},\dot{%
\zeta}\right) +L_{vx}\left( \dot{\xi},\zeta \right) +L_{xv}\left( \xi ,\dot{%
\zeta}\right) +L_{vv}\left( \xi ,\zeta \right) \right) dt.  \label{eq10}
\end{equation}%
For more details on this form see \cite{gon5}, Section 4.

\begin{proposition}
\label{propostion a2}Let $L$ be a Tonelli Lagrangian and $c$ be a cohomology
class with $\alpha \left( c\right) >e_{0}.$ Let us suppose that ${\widetilde{%
\mathcal{M}_{c}}}\left( L\right) $ has a unique minimizing measure supported
on a periodic orbit. Then there exists a $C^{\infty }$ 1-form $\eta $
(sufficiently close to zero) such that the perturbed Lagrangian $L+\eta -c$
has a unique minimizing measure supported on a hyperbolic periodic orbit $%
\Gamma $. Moreover the stable and unstable manifolds of $\Gamma $ intersect
transversally $W^{s}\left( \Gamma \right) \pitchfork W^{u}\left( \Gamma
\right) $.
\end{proposition}

\begin{proof}
We can consider $c=0.$ Let $\Gamma $ the minimizing periodic orbit in ${%
\widetilde{\mathcal{M}}}\left( L\right) .$ By the graph property, $\pi
|_{\Gamma }:\Gamma \rightarrow M,\pi \left( x,v\right) =x$ is injective, so $%
\pi \left( \Gamma \right) \subset M$ is a simple closed curve. We consider
coordinates on a tubular neighborhood of $\pi \left( \Gamma \right) $ in the
following way: $\varphi :U\rightarrow S^{1}\times
\mathbb{R}
^{n-1},\varphi =\varphi \left( x_{1},...,x_{n}\right) $ with $\varphi \left(
\Gamma \right) =S^{1}\times \left\{ 0\right\} $ and $\left\{ \frac{\partial
}{\partial x_{1}},\frac{\partial }{\partial x_{2}},...,\frac{\partial }{%
\partial x_{n}}\right\} $ is an orthonormal frame over the points of $\pi
\left( \Gamma \right) =\gamma .$

Given $\varepsilon >0,$we take the $C^{\infty }$ function $\lambda :$ $%
M\rightarrow
\mathbb{R}
,$ given by \ref{eq6} in the proof of Proposition \textit{\ref{proposicao1}}%
, as%
\begin{equation}
\lambda \left( x\right) =\left\{
\begin{array}{l}
0,\text{ on }\pi \left( \Gamma \right) \\
\frac{\varepsilon }{2}f\left( x\right) \left(
x_{2}^{2}+x_{3}^{2}+...+x_{n}^{2}\right) \text{ on }B\setminus \pi \left(
\Gamma \right) \\
0\text{ on }M\setminus B%
\end{array}%
\right. ,
\end{equation}%
where $B\subset U$ and $f$ is a non-negative bump function with support
contained in $B$ and which is one on a small neighborhood of $\pi \left(
\Gamma \right) .$ Then by Proposition \textit{\ref{proposicao1}} and
Equation \ref{eq7}, the 1-form
\begin{equation*}
\eta \left( x,v\right) =\eta _{x}\left( v\right) =\lambda \left( x\right)
\left( L_{v}\left( x,X\left( x\right) \right) -L_{v}\left( x,0\right)
\right) \left( v\right)
\end{equation*}%
is such that
\begin{equation*}
{\widetilde{\mathcal{M}}}\left( L+\eta \right) ={\widetilde{\mathcal{M}}}%
\left( L\right) =\Gamma .
\end{equation*}%
Note that $\eta $ can be made $C^{\infty }$ arbitrarily small. Now we define
$\tilde{L}=L+\eta .$ In order to aplly the index form \ref{eq10}, let us
calculate the derivatives $\tilde{L}_{vv},\tilde{L}_{vx}\ $and $\tilde{L}%
_{xx}$ on $\Gamma :$ we have $\tilde{L}_{vv}=L_{vv}$ and since $\partial
_{v}\eta \left( x,v\right) \left( h\right) =\eta _{x}\left( h\right) =\eta
\left( x,h\right) \forall h\in T_{x}M,$ we conclude that
\begin{equation*}
\tilde{L}_{vx}\left( h,k\right) =L_{vx}\left( h,k\right) +\left[ \partial
_{x}\eta \left( x,h\right) \right] \left( k\right) .
\end{equation*}

By taking the 1-form
\begin{equation*}
\omega \left( x,v\right) =\omega _{x}\left( v\right) =\left( L_{v}\left(
x,X\left( x\right) \right) -L_{v}\left( x,0\right) \right) \,\left( v\right)
,
\end{equation*}
we have $\eta _{x}=\lambda \left( x\right) \omega _{x}$. Then%
\begin{equation*}
\left[ \partial _{x}\eta \left( x,h\right) \right] \left( k\right) =d\lambda
\left( x\right) \left( k\right) \omega \left( x,h\right) +\lambda \left(
x\right) \left[ \partial _{x}\omega \left( x,h\right) \right] \left(
k\right) .
\end{equation*}%
In coordinates we have
\begin{equation}
d_{x}\lambda \left( x\right) \left( k\right) =\frac{1}{2}\varepsilon \left(
x_{2}^{2}+x_{3}^{2}+...+x_{n}{}^{2}\right) \left[ d_{x}f\left( x\right)
\left( k\right) \right] +f\left( x\right) \varepsilon \left(
0,x_{2},x_{3},...,x_{n}\right) \left( k\right) .  \label{eq12}
\end{equation}%
Since $\lambda |_{\pi \left( \Gamma \right) }=0$ we conclude that $%
d_{x}\lambda |_{\pi \left( \Gamma \right) }=0$ and $\partial _{x}\eta |_{\pi
\left( \Gamma \right) }=0.$ Therefore $\tilde{L}_{vx}=L_{vx}$ on $\Gamma .$

It ramains to calculate $\tilde{L}_{xx}.$ Observe that%
\begin{eqnarray*}
\tilde{L}_{xx}\left( x,v\right) \left( h,k\right) &=&L_{xx}\left( x,v\right)
\left( h,k\right) +\partial _{xx}\eta \left( x,v\right) \left( h,k\right) \\
&=&L_{xx}\left( x,v\right) \left( h,k\right) +d_{x}^{2}\lambda \left(
x\right) \left( h,k\right) \omega \left( x,v\right) +d_{x}\lambda \left(
x\right) \left( h\right) \partial _{x}\omega \left( x,v\right) \left(
k\right) \\
&&+d_{x}\lambda \left( x\right) \left( k\right) \partial _{x}\omega \left(
x,v\right) \left( h\right) +\lambda \left( x\right) \partial _{xx}\omega
\left( x,v\right) \left( h,k\right) .
\end{eqnarray*}%
Hence,
\begin{equation*}
\tilde{L}_{xx}\left( x,v\right) \left( h,k\right) =L_{xx}\left( h,k\right)
+d_{x}^{2}\lambda \left( x\right) \left( h,k\right) \omega \left( x,v\right)
\text{ on }\pi \left( \Gamma \right) .
\end{equation*}

Now it follows from \ref{eq12} that
\begin{equation*}
d_{x}^{2}\lambda \left( x\right) \left( h,k\right) =f\left( x\right)
\varepsilon \left[
\begin{array}{cc}
0 & 0 \\
0 & I%
\end{array}%
\right] \left( h,k\right) \text{ on }\pi \left( \Gamma \right) ,
\end{equation*}%
where $I$ is the identity matrix $\left( n-1\right) \times \left( n-1\right)
.$

Therefore, on $\Gamma ,$ we have $\tilde{L}_{vv}=L_{vv},\tilde{L}%
_{vx}=L_{vx}\ $and%
\begin{equation*}
\text{ }\tilde{L}_{xx}\left( x,v\right) \left( h,k\right) =L_{xx}\left(
h,k\right) +f\left( x\right) \varepsilon \left[
\begin{array}{cc}
0 & 0 \\
0 & I%
\end{array}%
\right] \left( h,k\right) \left( L_{v}\left( x,X\left( x\right) \right)
-L_{v}\left( x,0\right) \right) \left( v\right) .
\end{equation*}%
Now it is possible to compare the index of the original and the perturbed
lagrangian along the solution. Let $\tilde{I}_{T}$ and $I_{T}$ be the index
forms on $\left[ 0,T\right] $ for $\tilde{L}$ and $L,$ respectvely. Fix $%
\theta \in \Gamma $ and define $N\left( \theta \right) =\left\{ w\in T_{\pi
\left( \theta \right) }M|\left\langle w,\dot{\gamma}\right\rangle =0\right\}
.$ Hence $N\left( \theta \right) $ is generated by the vectors $\frac{%
\partial }{\partial x_{2}},...,\frac{\partial }{\partial x_{n}}.$ Denote $%
\pi _{N}\left( \xi _{1},\xi _{2},...,\xi _{n}\right) =\left( \xi
_{2},...,\xi _{n}\right) .$ Since the next steps of proof hold for Tonelli
Lagrangians in general, even in our case, they are entirely analogous to
proof of Theorem D in \cite{gon5} (Section 5, page 934). Therefore it is
suffices to proof that there exists $\delta >0,$ such that $\tilde{I}%
_{T}\left( \xi ^{T},\xi ^{T}\right) \geq I_{T}\left( \xi ^{T},\xi
^{T}\right) +\delta ,$ for certain vectors $\xi ^{T}$ satisfying $\left\vert
\pi _{N}\left( \xi ^{T}\left( t\right) \right) \right\vert >\frac{1}{2}$ for
every $0\leq t\leq \lambda $ and $T>T_{0}$. Indeed, in the coordinates $%
\left( x_{1},...,x_{n},\frac{\partial }{\partial x_{2}},...,\frac{\partial }{%
\partial x_{n}}\right) $ on $TU$ we have that%
\begin{eqnarray*}
\tilde{I}_{T}\left( \xi ^{T},\xi ^{T}\right) &=&\int_{0}^{T}\left( \tilde{L}%
_{vv}\left( \dot{\xi}^{T},\dot{\xi}^{T}\right) +2\tilde{L}_{vx}\left( \dot{%
\xi}^{T},\xi ^{T}\right) +\tilde{L}_{xx}\left( \xi ^{T},\xi ^{T}\right)
\right) dt \\
&=&\int_{0}^{T}\left( L_{vv}\left( \dot{\xi}^{T},\dot{\xi}^{T}\right)
+2L_{vx}\left( \dot{\xi}^{T},\xi ^{T}\right) +L_{xx}\left( \xi ^{T},\xi
^{T}\right) \right) dt \\
&&+\int_{0}^{T}\left( f\left( \gamma \right) \varepsilon \left[
\begin{array}{cc}
0 & 0 \\
0 & I%
\end{array}%
\right] \left( \xi ^{T},\xi ^{T}\right) \left( L_{v}\left( \gamma ,X\left(
\gamma \right) \right) -L_{v}\left( \gamma ,0\right) \right) \left( \dot{%
\gamma}\right) \right) dt \\
&\geq &I_{T}\left( \xi ^{T},\xi ^{T}\right) +\int_{0}^{T}\varepsilon
\sum_{i=2}^{n}\left( \xi _{i}^{T}\right) ^{2}Kdt\geq I_{T}\left( \xi
^{T},\xi ^{T}\right) +\varepsilon \frac{K\lambda }{4},
\end{eqnarray*}%
because $\left( L_{v}\left( \gamma ,X\left( \gamma \right) \right)
-L_{v}\left( \gamma ,0\right) \right) \left( \dot{\gamma}\right) \geq K$ (by
Lemma \ref{lema1}). Therefore, by taking $\delta =\frac{K\lambda }{4}$ we
obtain that $\Gamma $ is a hyperbolic periodic orbit for the Lagrangian $%
L+\eta .$ Now we must prove by perturbing $L+\eta $, if necessary, that the
stable and unstable manifolds intersect transversally $W^{s}\left( \Gamma
\right) \pitchfork W^{u}\left( \Gamma \right) .$ Actually, using similar
steps as above, the proof follows from the same arguments as in the proof of
Theorem D in \cite{gon5} (Section 5, page 934).
\end{proof}

Finally we can conclude the proof of Theorem \ref{teorema2} stated in
Introduction.

\begin{proof}
\textit{(of Theorem \ref{teorema2})} We aplly Theorem \ref{teo2} for $c=0$
to deduce that there exists a residual subset $\mathcal{O}$ of $\Gamma
^{\infty }\left( M\right) $ such that for any $\omega \in \mathcal{O}$, the
Lagrangian $L+\omega $ has a unique minimizing measure and this measure is
uniquely ergodic$.$ Let $\mathcal{A}$ be the subset of $\mathcal{O}$ of
1-forms for which the measure on $\mathfrak{M}\left( L+\omega \right) $ is
supported on a periodic orbit. Let $\mathcal{A}_{1}$ be the subset of $%
\mathcal{A}$ on which the minimizing periodic orbit is hyperbolic and its
stable and unstable manifolds intersect transversally $W^{s}\left( \Gamma
\right) \pitchfork W^{u}\left( \Gamma \right) $. The proof that $\mathcal{A}%
_{1}$ is relatively open on $\mathcal{A}$ and the final step are entirely
analogous to proof of Theorem D in \cite{gon5}. We repeat the final step
here only for the sake of completeness.

Let $\mathcal{U}$ be an open subset of $\Gamma ^{\infty }\left( M\right) $
such that $\mathcal{A}_{1}=\mathcal{A\cap U}$. Let $\mathcal{B}:=\mathcal{%
O\diagdown A}$. Since for an exact magnetic Lagrangian we have $\alpha
\left( c\right) >e_{0}$ for every cohomology class $c$ (see \cite{pat1},
Corollary 5.1), we can use Proposition \ref{propostion a2} to conclude that $%
\mathcal{A}_{1}$ is dense in $\mathcal{A}$. Therefore $\mathcal{A}_{1}\cup
\mathcal{B}$ is generic in $\Gamma ^{\infty }\left( M\right) .$ Let $%
\mathcal{V}=int\left( \Gamma ^{\infty }\left( M\right) \mathcal{\diagdown U}%
\right) .$ Hence $\mathcal{U\cup V}$ is an open and dense in $\Gamma
^{\infty }\left( M\right) .$ Since $\mathcal{A\subset }\overline{\mathcal{A}%
_{1}}\subset \overline{\mathcal{U}}$ we have $\mathcal{A\cap V\subset }%
\overline{\mathcal{U}}\mathcal{\cap V=\emptyset }$, that is $\mathcal{A\cap
V=\emptyset }$. Moreover $\mathcal{O=A\cup B}$ is generic and%
\begin{eqnarray*}
\left( \mathcal{U\cup V}\right) \cap \left( \mathcal{A\cup B}\right) &=&%
\left[ \left( \mathcal{U\cup V}\right) \mathcal{\cap A}\right] \cup \left[
\left( \mathcal{U\cup V}\right) \cap \mathcal{B}\right] \\
&=&\left( \mathcal{U}\cap \mathcal{A}\right) \mathcal{\cup }\left[ \left(
\mathcal{U\cup V}\right) \cap \mathcal{B}\right] \\
&\subset &\mathcal{A}_{1}\mathcal{\cup B}\text{.}
\end{eqnarray*}%
This shows that $\mathcal{A}_{1}\mathcal{\cup B}$ is generic in $\Gamma
^{\infty }\left( M\right) $.
\end{proof}

\section{Acknowledgements} 
I am grateful to Mário Jorge Dias Carneiro and José Antônio Gonçalves Miranda for several helpful comments and suggestions. I thank also to FAPEMIG-BRAZIL which supported partially this work.

\end{document}